\documentclass[a4paper]{amsart}

\usepackage{color}
\usepackage{amsxtra}
\usepackage{mathtools}
\usepackage{enumerate}
\usepackage{blindtext}
\usepackage[inline]{enumitem}
\usepackage{nicefrac}
\mathtoolsset{showonlyrefs} 
\usepackage{ mathrsfs }
\usepackage{tikz}

\usetikzlibrary{shadings}
\usepackage{color}

\usepackage{hyperref}
\hypersetup{
  colorlinks   = true, 
  urlcolor     = blue, 
  linkcolor    = blue, 
  citecolor   = red 
}
\xdefinecolor{darkgreen}{RGB}{175, 193, 36}
\newtheorem{theorem}{Theorem}[section]
\newtheorem{lemma}{Lemma}[section]

\newtheorem{definition}{Definition}[section]
\newtheorem{corollary}{Corollary}

\title[Existence of supersolutions for nonlocal equations]{Study of the existence of supersolutions for nonlocal equations with gradient terms}

\author[B. Barrios]{Bego\~na Barrios}
	\address{Bego\~na Barrios \hfill\break\indent
		Departamento de An\'{a}lisis Matem\'{a}tico,
		Universidad de La Laguna\hfill 
		\break \indent C/. Astrof\'{\i}sico Francisco S\'{a}nchez s/n, 
		38200 -- La Laguna, SPAIN}
	\email{bbarrios@ull.es}

\author[L. M. Del Pezzo]{Leandro M. Del Pezzo}
	\address{Leandro M. Del Pezzo \hfill\break\indent
		UBA -- CONICET \hfill\break\indent
		Departamento de Matem\'aticas, Facultad de Ciencias Exactas y Naturales \hfill\break\indent
		Universidad de Buenos Aires 
		\hfill\break\indent Pabell\'on I -- Ciudad Universitaria  (C1428BCW)
		\hfill\break\indent Buenos Aires, ARGENTINA. }
	\email{ldpezzo@dm.uba.ar}
	\urladdr{http://cms.dm.uba.ar/Members/ldpezzo/}

\begin{document}

\begin{abstract}
	We study the existence of positive supersolutions of nonlocal equations of type
	$(-\Delta)^s u+ |\nabla u|^q=\lambda f(u)$ posed in exterior domains where the 
	datum $f$ can be comparable with $u^{p}$ near the origin. We prove that 
	the existence of bounded supersolutions depends on the values of $p$, 
	$q$ and $s$.



\end{abstract}

\maketitle

\section{Introduction}

	To establish nonexistence results for positive solutions of nonlinear 
	equations has been object of study by many authors in the last decades.
	{This is largely due, apart from
	its own interest,} by its applications like proving a priori 
	bounds for positive solutions (see for instance \cite{MR619749}) or 
	studying the singularities of such solutions (see \cite{MR2350853}). The 
	first kind of nonexistence results, commonly called Liouville's type 
	results, was obtained by Gidas and Spruck for the seminal equation 
	\[
		-\Delta u= u^{p}\quad\mbox{ in }\mathbb{R}^{N},
	\]
	see \cite{MR615628}. 
	
	Later on, the question of 
	nonexistence of positive solutions for equations that involve  
	differential operators, different than the Laplacian, was 
	{ahead} by several authors, for example, { in } 
	\cite{MR1077263,MR1321809,MR1004713,MR1753094,MR2569325}. 
	Recently, 
	Armstrong and Sirakov, in \cite{MR2905384}, 
	proved a very general Liouville{'s} type theorem for 
	viscosity supersolutions of the equation
	\[
		M u= f(u)\quad\mbox{in }\mathbb{R}^{N}\setminus B_{R_0},
	\]
	where $R_0>0$, $M$ is a general fully nonlinear operator and $f$ is a real  
	positive function in $(0,\infty)$ satisfying 
	\[
		\lim_{t\to 0^{+}}\frac{f(t)}{t^{p^*}}>0,\quad p^*:=\frac{N}{N-2}.
	\] 

	This result{,} in particular{,} implies that the equation $-\Delta u= u^{p}$ 
	does not admit any positive supersolutions if $0<p\leq p^*$ 
	when it is posed, not only in the whole Euclidean space $\mathbb{R}^{N}$, but 
	also in an exterior domain. The result obtained in 
	\cite{MR2905384}  is based on a Hadamar type property of the solutions{,} 
	which requires that the differential operator $M$ is homogeneous. {Then,} it is 
	natural to ask if this kind of nonexistence result can be obtained if  
	the operator is not homogeneous. In fact, this was the main objective of 
	the recent work of Alarc\'on, Garc\'ia-Meli\'an and Quaas 
	\cite{MR3039209} where a general nonexistence result 
	for  {positive} supersolutions of 
	\begin{equation}\label{local}
		-\Delta u+|\nabla u|^q=\lambda f(u)\,\text{ in } 
		\mathbb{R}^N\setminus B_{R_0},\, \lambda>0,
	\end{equation}
	depending of the range of $p$, $q$ and $\lambda$ was proved. 
	We have to mention that other authors addressed this problem  
	but {establishing} results related with the existence or nonexistence of 
	{\it radial} supersolutions {before}
	(see for instance \cite{MR1188490, MR1422294}). We also notice that the presence of the 
	gradient term introduces the possibility of having the existence of supersolutions of 
	\eqref{local} that either are bounded or diverged at infinite because 
	the equation is posed in an exterior domain. To complete the study of 
	\eqref{local}, in the more general framework of 
	fully nonlinear differential operators, Rossi in \cite{MR2358359} proved the nonexistence of 
	supersolutions that do not blow up at infinity of 
	\begin{equation}\label{lin}
		-\Delta u+ b(x)|\nabla u|= c(x) u\quad \text{ in } 
		\mathbb{R}^N\setminus B_{R_0},
	\end{equation}
	when $b(x)$ and $c(x)$ are bounded functions. Finally we also show up 
	that in \cite{MR3046986} the nonexistence of positive supersolutions of 
	\eqref{lin} for general unbounded weights $b(x)$ and $c(x)$,
	was obtained using a completely different approach than in 
	\cite{MR2358359}.

	As far as we know, all the previous results commented above are the most general Liouville's type theorems established for supersolutions for the equation \eqref{local} up to now. 
	Thus, the situation can be summarized in the 
	{ Figure \ref{figuralocal}}. 
	
	\begin{figure}[h]
	\begin{center}
		\begin{tikzpicture}
		    \draw[black!20,fill=red!40,scale=2] 
	          (1.33333,0) -- (1.33333,2)-- (2.5,2) -- (2.5,0)--(1.33333,0);
	        \draw[black!20,fill=red!40,scale=2] 
				(1,0) -- (1,1)-- (1.1,1.22222)-- (1.2,1.5) -- 
				(1.33333,2)-- (1.33333,0)--(1,0);
            \draw[->,scale=2] (0,0) -- (2.5,0) node[right] {$q$};
            \draw[->,scale=2] (0,0) -- (0,2.5) node[left] {$p$};
		    \draw[thick,densely dotted,scale=2](0,1) node[left] {$1$} 
		    -- (1,1);
		    \draw[red, line width=2pt, scale=2] (1,0) node[below] {$1$} -- (1,1) ;
		    \draw[thick,densely dotted,scale=2] (0,2) node[left] 
		       {$\tfrac{N}{N-2}$} -- (1.33333,2) -- (1.33333,0);
		    \draw[scale=2] (0,2) (1.3,0) node[below] 
		       {\quad$\tfrac{N}{N-1}$} ;
		    \draw[red,line width=2pt, scale=2] (1.33333,2)--(2.5,2);
		    \draw[blue,scale=2] (.8,1.5) node {$p=\frac{q}{2-q}$};
		    \draw[blue,line width=2pt,scale=2] plot[domain=1:1.33333] (\x,{\x/(2-\x)});
		    \draw[black,fill=white,scale=2] (1.33333,2) circle (.2ex);
		    \draw[black,fill=red,scale=2] (1,1) circle (.2ex);
		            	             
    	\end{tikzpicture}
	\end{center}
	\label{figuralocal}
	\caption{
		 In the red zone, {the problem} \eqref{local} does 
		not admit positive supersolutions which do not blow up at infinity. 
		On the blue curve a bounded positive supersolution exists
		depending on the value of 
		$\lambda$ (see \cite[Theorems 1,2,3]{MR3039209} and \cite{MR2358359})
		}
    \end{figure}
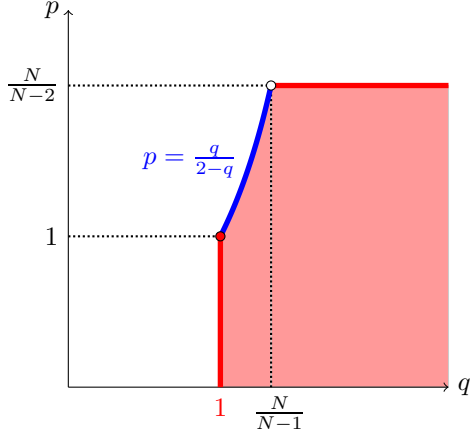
    \medskip 
    
	Motivated by the previous results our objective in the work at hands is to 
	study the Liouville's type result for supersolutions of the nonlocal 
	equation
	\begin{equation}\label{main}
		(-\Delta)^s u+ |\nabla u|^q=\lambda f(u) \text{ in } 
		\mathbb{R}^N\setminus B_{R_0},\tfrac{1}2<s<1,\, {\lambda>0},
	\end{equation}	
	where $(-\Delta)^s$, $0<s<1$, is the well-known 
	\emph{fractional Laplacian}, defined, on smooth functions as
	\begin{equation}\label{operador}
		(-\Delta)^s u(x) \coloneqq C_{N,s}\int_{\mathbb{R}^N} 
		\frac{u(x)-u(y)}{|x-y|^{N+2s}} dy,
	\end{equation}
	with $C_{N,s}$ a normalization constant that is usually omitted for 
	brevity. The integral in \eqref{operador} 
	has to be understood in the principal value sense, that is, as the limit 
	when $\epsilon\to 0$ of the same integral taken in 
	$\mathbb{R}^N\setminus B_\epsilon(x).$ 
	
	It is clear that the 
	fractional Laplacian is well defined for functions that belong, for 
	instance, to $\mathcal{L}_{2s}\cap \mathcal{C}^{2}$ where
	\begin{equation}\label{L2s}
		\mathcal{L}_{2s}\coloneqq\left\{u\colon\mathbb{R}^{N}\to 
		\mathbb{R}\colon 
		\int_{\mathbb{R}^{N}}\frac{|u(x)|}{1+|x|^{N+2s}}<\infty\right\}.
	\end{equation}
	Alternatively, the fractional Laplacian can be written in a more useful way as the following integral 	
	\[
		(-\Delta)^s u(x) \coloneqq C_{N,s}\int_{\mathbb{R}^N} \dfrac{2u(x)-u(x+y)-u(x-y)}{|y|^{N+2s}}dy,
	\]
	that is absolutely convergent for any $u\in \mathcal{L}_{2s}\cap \mathcal{C}^{2}.$ 
	
	Problems with non local diffusion that involve the 
	fractional Laplacian operator and other integro-differential operators 
	have been intensively studied in the last years. {These} nonlocal operators appear when 
	we 
	model different physical situations as anomalous diffusion and 
	quasi-geostrophic flows, turbulence and
	water waves, molecular dynamics and relativistic quantum mechanics of 
	stars (see \cite{MR1081295,MR2196360} and {the references therein}). They also appear 
	in mathematical finance \cite{MR2512800,MR2042661}, elasticity 
	problems \cite{MR0120846}, obstacle problems 
	\cite{MR3861075,MR3783214}, 
	phase transition problems \cite{MR1612250} and crystal dislocation structures,	
	\cite{MR1442163} among others.

	However, our particular interest in the study of \eqref{main} comes not 
	so much for a possible direct application but from the nature of the 
	problem itself. 
	As far as we know, there is not too much 
	known about general nonexistence results for equations that mix local 
	and nonlocal operators like \eqref{main}. We want to comment that other general 
	Liouville's type results for nonlocal elliptic equations can be found in, for instance, 	
	\cite{MR3918222,MR2759038,MR3511811,2014arXiv1401.7402Z}, 
	{but}, as we said, 
	non local equations with gradient terms have not been deeply studied, (see for 
	example, \cite{MR3177345,MR3129851,MR3043588,MR3017289,MR3054309}).
	 
	Finally, we also want to highlight that similar results as those given in \cite{MR1188490} for the 
	corresponding problem \eqref{main} posed in all the space $\mathbb{R}^{N}$, has not been obtained in the 
	nonlocal framework due to the complicate expression of the fractional Laplacian for radial functions (see 
	for instance \cite{MR2917408}) until today. 

	\medskip

	Following some ideas developed in \cite{MR3039209,MR1753094,MR2739791}{,} 
	our objective is to prove the existence, depending on the range of $p$, 
	$q$ and $\lambda$, of positive supersolutions of \eqref{main} where $f$ 
	satisfies
 	\begin{equation}\label{eq:f1}
		\liminf_{t\to0^+} \dfrac{f(t)}{t^p}>0.
	\end{equation}

	As occurs in the local case, the {positive supersolutions} 
	could be 
	bounded or could diverge at infinity due to the fact that the problem 
	is posed in an exterior domain. More precisely, by proving some upper and lower bounds for an auxiliar 	
	function 
	\[
		m(R_1,R_2)\coloneqq\min\left\{u(x)\colon x\in A(R_1,R_2),\, 
		0<R_1<R_2\right\},
	\]
	where $A(R_1,R_2)\coloneqq\left\{x\colon R_1\le |x|\le R_2\right\}$ (see 
	Section \ref{nonexist}), we are able to get some relevant nonexistence results 
	for {positive} supersolutions 
	{as the following results establish.}        

	\begin{theorem}\label{theorm:subcritical1}
		    Let $N>2s>1.$ Assume that $f\colon(0,\infty)\to{ (0,\infty)}$ is a 
		    continuous function
		    verifying \eqref{eq:f1}. 
		    If 
			   $$1<q<\displaystyle\frac{N}{N+1-2s} \mbox{ 
			    and } 0<p<\dfrac{(2s-1)q}{2s-q},$$
			    or 
			    $$ q\ge\displaystyle\frac{N}{N+1-2s} 
			   \mbox{ and } 0<p<\displaystyle\frac{N}{N-2s},$$ 
		    then there are no positive supersolutions to \eqref{main}
	which do not blow up at infinity.
		\end{theorem}			
	
	\begin{theorem}
		\label{theorem:cc}
		 Let $N>2s>1.$ Assume that {$f\colon(0,\infty)\to(0,\infty)$} is a 
		 continuous {nondecreasing} function
		 verifying \eqref{eq:f1}.  If $p=\tfrac{N}{N-2s}$, and 
		$q>\tfrac{N}{N+1-2s}$ then there are no positive 
		supersolution of \eqref{main}
		which do not blow up at infinity.
	\end{theorem}

	That is, we obtain the equivalent result as 
	\cite[Theorem 1, Theorem 3]{MR3039209} in the nonlocal framework 
	(i.e. compare the red zones in { Figures \ref{figuralocal} and \ref{figura}}).       
       
	In relation to the existence of supersolutions that do not diverge 
    at infinity, doing a carefully analysis we obtain the { following.} 
    	\begin{theorem}\label{theorem:cc1}
        Let $N>2s>1.$ Then there exists a positive bounded supersolution of 
         		\begin{equation}
	        		\label{eq:cc1}
	         		(-\Delta)^s u + |\nabla u|^q=\lambda u^p,\quad \mbox{in $\mathbb{R}^{N}\setminus B_{R_0}$, $\lambda>0,$}
         		\end{equation}
		if one of the following cases hold
        \renewcommand{\theenumi}{\roman{enumi}}%
        \begin{enumerate}[topsep=8pt,itemsep=4pt,partopsep=4pt, parsep=4pt]
       		\item $p=\displaystyle\frac{N}{N-2s}$ and $0<q<1;$
       		\item $p=\displaystyle\frac{N}{N-2s},$ $q=1$ and $R_0>\tilde{R}_0$ for some $
       		\tilde{R}_{0}$ big enough;
       		\item $p=\displaystyle\frac{N}{N-2s},$ $1<q\le\displaystyle\frac{N}{N+1-2s}$ and $
       		\lambda\in(0,\lambda_0)$ for some 
       		$\lambda_0>0;$
       		\item $\displaystyle\frac{2s-1}{2s-q}q\le p<\displaystyle\frac{N}{N-2s},$ 
       		$\displaystyle\frac{N+2s}{N+1}<q<\displaystyle\frac{N}{N+1-2s},$ and 
       		$\lambda\in(0,\lambda_0)$ for some 
       		$\lambda_0>0;$
       		\item $\displaystyle\frac{N+2s}{N+2s-q}q\le p<\displaystyle\frac{N}{N-2s},$ 
       		$0<q<\displaystyle\frac{N+2s}{N+1},$ $q\neq1$ and 
       		$\lambda\in(0,\lambda_0)$ for some 
       		$\lambda_0>0.$ In the case of $q=1$, $R_0$ should be bigger 
       		than $\tilde{R}_0$ for some $\tilde{R}_{0}$ big enough;
       		\item $\displaystyle\frac{N+2s}{N}< p<\displaystyle\frac{N}{N-2s},$ 
       		$q=\displaystyle\frac{N+2s}{N+1},$ and 
       		$\lambda\in(0,\lambda_0)$ for some 
       		$\lambda_0>0.$
		\end{enumerate}
         
    \end{theorem}    

	We show up that to the contrary of what happens in the local case, radial reductions 
	are not useful to study our equation (see for instance \cite{MR3541501} 
	and references therein). Moreover, the function $m$ does not 
	satisfy $$m(R_1,R_2)=\min\{m(R_1), m(R_2)\},$$ where { $m(R)=min\{u(x)\colon {|x|=R}\}$} 
	is, in the local case, a monotone function for $R>R_1$, (see \cite[Lemma 1]{MR3039209}). 
	This and the fact that it is really complicate to work with the
	explicit value of general radial functions of the nonlocal operator, 
	force us to study different cases specially when $p$ is critical.  For that, optimal estimates for the fractional Laplacian of some 
	particular radial functions will be needed (see Lemmas \ref{lema.BV}, 
	\ref{lemma:subsolution2}, \ref{lemma:aux4} and the 
	proof of Theorem \ref{theorem:cc}). Due to the complications 
	that the nonlocal operator introduced, we left as an open problem the existence 
	of bounded supersolutions in the range 
	$$
		\frac{2s-1}{2s-q}q\le p<\frac{N}{N+2s-q}q,\quad 0<q\le\frac{N+2s}{N+1},
	$$
	{because} in this case it is not clear how can we find a suitable function $\psi$ such that 
	$|(-\Delta)^s \psi|$ decays faster than $|\nabla \psi|^q$ at infinity 
	(see {\it Cases 4} and {\it 5} in the proof of Theorem \ref{theorem:cc1} and Figure \ref{figura}).

	We notice that the essential feature which continuously appears along the main proofs of our work is 
	a comparison principle that in the local framework was proved, in, 
	for instance, \cite[Theorem 10.1]{MR737190}.   
	
	Finally we want also to mention that, in the supercritical case, that is when $N>2s$ and 
    $p>\tfrac{N}{N-2s}$, Felmer and Quaas in \cite[proof of Theorem 1.3]{MR2739791} 
    have recently shown that, for any 
    $\beta\in\left(\displaystyle\frac1{p-1},\displaystyle\frac{{N-2s}}{2s}\right)$ 
    there exists $C>0$ such that $$v_{\beta}(x)\coloneqq C(1+|x|)^{-2s\beta},$$
     is a bounded supersolution of $(-\Delta)^s u= u^p \text{ in }\mathbb{R}^N.$	 
     Therefore, trivially, $v_\beta$ is also 
    a supersolution of \eqref{main} when $\lambda=1$ and 
    $f(u)=u^p$ for all $q\ge0.$ So that, we can also formulate the next
			
	\begin{theorem}\label{theorm:HolesSuperCC}
		 Let $N>2s,$ $p>\tfrac{N}{N-2s}$ and $q\ge0.$ Then  there exists a classical positive 
		 and bounded supersolution of $(-\Delta)^s u+|\nabla u|^q = \lambda u^p$ in $\mathbb{R}^N\setminus B_{R_0}$, for $\lambda\leq 1$.
	\end{theorem}      
		
	The conclusion obtained in the Theorems  \ref{theorm:subcritical1}, \ref{theorem:cc}, \ref{theorem:cc1} 
	and \ref{theorm:HolesSuperCC} for $f(u)=\lambda u^p$, $\lambda\leq 1$, can be summarized in the following graphic 
	where 
	$$
		q_1=\frac{N+2s}{N+1}\quad\mbox{ and }\quad q_2=\frac{N}{N+1-2s}.
	$$
	\begin{figure}[h]
		\begin{center}
	            \begin{tikzpicture}
	             \draw[black!20,fill=purple!30,scale=3] 
	                 	(0,1.6) -- (0,2.2)
	                 	-- (2,2.2) -- (2,1.6);
	                 \draw[black!20,fill=red!40,scale=3] 
	                 	(1.14285714286,0) -- (1.14285714286,1.6)
	                 	-- (2,1.6) -- (2,0)--(1.14285714286,0);
	                  \draw[black!20,fill=red!40,scale=3] (1,0) -- 
	                  	(1,1)-- (1.1,1.375)-- (1.14,1.58333333333) 
	                  	-- (1.14285714286,1.6)-- (1.14285714286,0)
	                  	--(1,0);
	                   \draw[black!20,fill=blue!50,scale=3] 
	                 	(0,1.6)--(1.14285714286,1.6)-- (1.1,1.375) -- (1,11/9) -- (3/4,33/38)
	                 	-- (1/2,11/20)--(0.25,0.2619047619) -- 
	                 	(0,0)-- (0,1.6); 
                    \draw[->,scale=3] 
                    	(0,0) -- (2,0) node[right] {$q$};
                    \draw[->,scale=3] 
                    	(0,0) -- (0,2.2) node[left] {$p$};
		             \draw[scale=3](1,0) node[below] {$1$};
		            \draw[red,dashed, scale=3] 
		            	(1,0)  -- (1,1) ;
		            \draw[thick,densely dotted,scale=3] 
		            	(1,1) -- (1,11/9) ;
		            \draw[thick,line width=2pt,color=darkgreen,scale=3] 
		            	(1,11/9) -- (1,1.6) ;
		           \draw[thick,densely dotted,scale=3] 
		           		(0,1.6) node[left] {$\tfrac{N}{N-2s}$} 
		           		-- (1.14285714286,1.6) -- (1.14285714286,0) ;
		           \draw[thick,densely dotted,scale=3] 
		           		(0,1.375) node[left] {$1+\tfrac{2s}{N}$} 
		           		-- (1.1,1.375) -- (1.1,0) node[below] 
		             {$q_1$}  ;
		           \draw[thick,densely dotted,scale=3] 
		           		(0,11/9) node[left] {$1+\tfrac{1}{N+2s-1}$} 
		           		-- (1,11/9);
		           \draw[scale=3]  (1.2,0) node[below] 
		             {$q_2$} ;
		           \draw[red!40,line width=2pt,scale=3] 
		           		(1.14285714286,1.6)--(2,1.6);
		            \draw[blue!50,line width=2pt,scale=3] 
		            	plot[domain=1.1:1.14285714286] (\x,{\x/(3-2*\x)});
		            \draw[blue,dashed,scale=3] 
		            	plot[domain=0:1.1] (\x,{\x/(3-2*\x)});

		             \draw[blue,scale=3] 
		            	plot[domain=0:1.1] ((\x,{(11)/(11-2*\x)*\x});
		            \draw[blue,very thick,line width=2pt,scale=3] 
		            	(1,1.6)--(1.14285714286,1.6);
		            \draw[purple!30,very thick,line width=2pt,scale=3] 
		            	(0,1.6)--(1,1.6);
		            
		          	\draw[black,fill=darkgreen,scale=3] 
		             	(1,1.6) circle (.1ex);
		             \draw[black,fill=darkgreen,scale=3] 
		             	(1,11/9) circle (.1ex);
		             \draw[black,fill=blue,scale=3] 
		             	(1.14285714286,1.6) circle (.1ex);
		            
		             \draw[black,fill=white,scale=3] 
		             	(1.1,1.375) circle (.1ex);
		             \draw[blue] (2.2,1) node {$p=\tfrac{2s-1}{2s-q}q$};
                \end{tikzpicture}
                \caption{	
                	In the red zone there is not any positive supersolution which 
                	do not blow up at infinite. In the pink one 
                	(supercitical case $p>\tfrac{N}{N-2s}$ 
					and critical $p=\tfrac{N}{N-2s}$, with $0<q<1$)	
					no extra restrictions on $\lambda$ or $R_0$ 
					to prove the existence of bounded supersolutions is needed. 
					The blue region corresponds to the cases in which the existence of bounded 
					supersolutions depends on some upper bound on the parameter $\lambda$ 
					(see cases {\it iii)-vi)} of Theorem \ref{theorem:cc1}). 
					Finally, the green { line} shows the situation obtained when $q=1$ 
					where the parameter $R_0$ has to be big enough.
					}
                 \label{figura}
            \end{center}

       \end{figure}
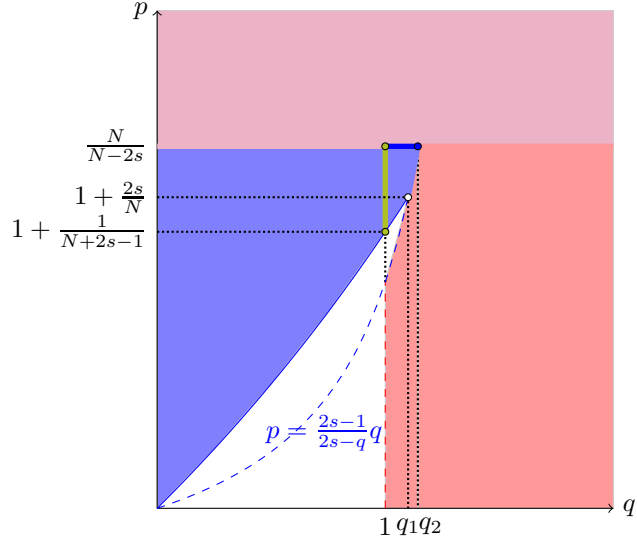
	

\section{Some preliminaries}\label{SP}

    In this section, we gather some preliminary properties and 
    definitions which will 
    be useful in the forthcoming ones. 
    
    \subsection{Classical solution}

	\begin{definition}\label{clasica}
		Let $s\in\left(\tfrac{1}{2},1\right),$ $q,p,R_0>0,$  
    	$B_{R_0}$ be the $N-$ball with radius $R_0$ and center $0,$ and 
    	$f\colon \mathbb{R} \to\mathbb{R}$ be a function. 
    	We say that
    	$u\in \mathcal{C}^2(\mathbb{R}^N\setminus B_{R_0})\cap\mathcal{L}_{2s}$ is 
    	a {classical supersolution} 
    	(resp. {subsolution}) 
    	of \eqref{main} if
    	\[
    		(-\Delta)^s u(x)+ |\nabla u(x)|^q
    		\ge (\text{resp.}\le) \, f(u) \text{ in } 
    		\mathbb{R}^N\setminus B_{R_0},
    	\]
		where $\mathcal{L}_{2s}$ was given in \eqref{L2s}.	
		We also call $u$ a {classical solution} of \eqref{main} 
		if it is both a sub- and supersolution.
	\end{definition}
	
    	
    %
	%
	Throughout this article we will be always dealing with positive classical supersolutions. 
	However, as we commented in the introduction, 
	we will not make any assumption about if these supersolutions are 
	bounded or not. In fact, we will distinguish between those that are 
	bounded and others that not.

    \subsection{Particular subsolutions}
    	{We show up now some results regarding with the existence of 
    	subsolutions when the datum $f$ is equal to zero. Before that we give 
    	a useful pointwise inequality that will be also the main tool to 
    	prove the existence of bounded supersolutions of \eqref{main} (see Section \ref{exsol}). The proof of it can be found in \cite[Lemma 2.1]{MR3122168}.
		
		\begin{lemma}\label{lema.BV}
			Let $\varphi\in \mathcal{C}^2(\mathbb{R}^N)$ be a positive real 
			function that is radially symmetric and
			decreasing for every $|x|>1.$ 
			Assume also that $\varphi(x)\le |x|^{-\sigma}$, $|\nabla \varphi(x)|\leq c_0|x|^{-\sigma-1}$ and 
			$|D^2\varphi(x)|\le c_0|x|^{-\sigma-2}$ for some 
			$\sigma>0$ and $|x|$ large enough.	
			Then there exist $\widetilde{R}>1$ and some 
			positive constants $c_1,c_2,c_3$ that depend only on 
			$\sigma,s,N,$ and $\|\varphi\|_{C^{2}(\mathbb{R}^2)}$, 
			such that  
			\[
				|(-\Delta)^s\varphi(x)|\le
				\begin{cases}
					\dfrac{c_1}{|x|^{\sigma+2s}} &\text{if }\sigma<N,\\[8pt]
					\dfrac{c_2\log(|x|)}{|x|^{N+2s}}&\text{if }\sigma=N,
					\\[8pt]
					\dfrac{c_3}{|x|^{N+2s}}&\text{if }\sigma>N,
				\end{cases}
			\]
			for every $|x|>\widetilde{R}>1$.	
			For $\sigma>N$ the reverse estimate holds from below, that is,
			if $\varphi\ge0$ there exists a positive constant $c_4$ 
			such that
			\[
				|(-\Delta)^s\varphi(x)|\ge \dfrac{c_4}{|x|^{N+2s}},
			\] 
			for all $|x|>\widetilde{R}>1.$
		\end{lemma}

       For any $q>1$ let us consider now the equation
        \begin{equation}
            \label{eq:subsolution}
                (-\Delta)^s u(x) +|\nabla u|^q=0
                \quad\text{in } \mathbb{R}^N\setminus B_R,
        \end{equation}
        for $R$ large enough. Using the previous result we get the next
        \begin{lemma}\label{lemma:subsolution1}
            Let $1<2s<N,$ $\phi\in \mathcal{C}^2(\mathbb{R}^N)$ be as in 
            Lemma \ref{lema.BV}. If 
            \begin{align*}
            	1<&q<\dfrac{N+2s}{N+1} \quad\text{ and }
            	\quad \sigma\ge \dfrac{N+2s}q-1>N,\quad\text{or} \\
            	&q=\dfrac{N+2s}{N+1} \quad\text{ and } \quad\sigma>N,
            \end{align*}
            then
            for any $A>0$ small enough the function $\phi_A(x)=A\phi(x)$
            is a classical subsolution of \eqref{eq:subsolution}. 
        \end{lemma}
    
        \begin{proof}
            If $\sigma>N$, using Lemma \ref{lema.BV} it follows that
            \[
                (-\Delta)^s\phi(x)\le c(N,s,\sigma) |x|^{-N-2s}\quad
                \text{in }\mathbb{R}^N\setminus B_{\widetilde{R}},
            \]
            with $ c(N,s,\sigma)<0.$
            Thus
            \[
                (-\Delta)^s \phi_A(x)+|\nabla \phi_A(x)|^q
                \le  c(N,s,\sigma) A |x|^{-N-2s}+A^qc_0^q|x|^{-(\sigma+1)q}
                \quad\text{in }\mathbb{R}^N\setminus B_{\widetilde{R}},
            \] Hence, if $\displaystyle1<q<\tfrac{N+2s}{N+1}$ 
            and $\sigma\ge \tfrac{N+2s}{q}-1$
            or if $q=\tfrac{N+2s}{N+1}$ and $\sigma>N$
            then 
             \[
                (-\Delta)^s \phi_A(x)+|\nabla \phi_A(x)|^q
                \le (c(N,s,\sigma) A+A^qc_0^q) |x|^{-N-2s}\le0
                \quad\text{in }\mathbb{R}^N\setminus B_{\widetilde{R}},
            \]
            as long as $A$ is small enough.
        \end{proof}
    
    	Choosing a particular function in the class of those referred in Lemma \ref{lema.BV}, 
    	we can also get the existence of a 
    	subsolution of the homogeneous equation \eqref{eq:subsolution} 
    	for $q>\tfrac{N+2s}{N+1}$ for a suitable range of the positive parameter $\sigma$. 
    	Indeed, we have the following
        \begin{lemma}\label{lemma:subsolution2}
            Let $1<2s<N.$ If 
            \begin{align*}
            	\frac{N+2s}{N+1}<&q<\frac{N}{N+1-2s}
            	\quad\text{ and }\quad
            	N>\sigma\ge\frac{2s-q}{q-1}>0, \text{ or }\\
            	&q\ge\frac{N}{N+1-2s}\quad \text{ and }\quad N>\sigma>N-2s>0,
            \end{align*}
			then for all $A,\varepsilon>0$ 
	        small enough
	        \[
                w_{\sigma.\varepsilon}^A(x)\coloneqq
                \begin{cases}
                    A\varepsilon^{-\sigma} &\text{if } |x|\le\varepsilon,\\
                    A|x|^{-\sigma}&\text{if } |x|>\varepsilon,
                \end{cases}
            \]
            is a classical subsolution of \eqref{eq:subsolution}.
        \end{lemma}
        \begin{proof}
            By \cite[Lemma 4.1]{FALL2018} (see also \cite[Remark 4.7 (iii)]{BQ} and \cite{MR2425175}) 
            we get that
            \[  
                v_\sigma(x)\coloneqq |x|^{-\sigma} \quad \sigma\in(-2s,N),
            \]
            is a classical solution of
            \[
                (-\Delta)^s v(x)=\gamma_\sigma |x|^{-2s}v\quad
                \text{in }\mathbb{R}^N\setminus\{0\},
            \]
            where
            \[
                \gamma_\sigma=2^{2s}
                \dfrac{\Gamma\left(\dfrac{N-\sigma}{2}\right)
                \Gamma\left(\dfrac{2s+\sigma}{2}\right)}
                {\Gamma\left(\dfrac{\sigma}{2}\right)
                \Gamma\left(\dfrac{N-2s-\sigma}{2}\right)},
            \]
            is a concave function that is negative if $\sigma\in (-2s,0)\cup(N-2s,N)$ (see \cite[Corollary 3.1]{MR2739791}).        
            Taking $\varepsilon>0$ and 
            \[
                w_{\sigma.\varepsilon}(x)\coloneqq
                \begin{cases}
                    \varepsilon^{-\sigma} &\text{if } |x|\le\varepsilon,\\
                    v_{\sigma}(x)&\text{if } |x|>\varepsilon,
                \end{cases}
            \]
             we get
            \begin{align*}
                (-\Delta)^s w_{\sigma.\varepsilon}(x) &= (-\Delta)^s v_\sigma(x)
                +\int_{B_\varepsilon}\dfrac{|y|^{-\sigma}-\varepsilon^{-\sigma}}{|x-y|^{N+2s}} dy\\
                &\le\gamma_\sigma |x|^{-\sigma-2s}+ C\dfrac{\varepsilon^{-\sigma+N}}{(|x|-\varepsilon)^{N+2s}},
            \end{align*}
for every $|x|\gg 1$. Then, for any $\sigma\in (N-2s,N),$ $\varepsilon$ small enough and $|x|\gg1$ we have
            \[
                (-\Delta)^s w_{\sigma.\varepsilon}(x) \le \gamma_{\sigma,\varepsilon} |x|^{-\sigma-2s},
            \]
            where $\gamma_{\sigma,\varepsilon}=\gamma_\sigma +C\varepsilon^{-\sigma+N}<0.$
            We observe now that
          	\[
                \sigma+2s\le (1+\sigma)q\Longleftrightarrow \sigma\ge\dfrac{2s-q}{q-1}\eqqcolon \alpha(q),
          	\]
          	and, moreover, the function $\alpha(q)$ is decreasing in $(1,\infty)$ and satisfies
          	\[
            	 \alpha\left(\frac{N+2s}{N+1}\right)=N\quad\text{and}\quad  
                \alpha\left(\frac{N}{N+1-2s}\right)=N-2s.
          	\] 
			Hence, under our hypothesis, 
            \[ 
                \begin{aligned}
                    (-\Delta)^s w_{\sigma.\varepsilon}^A(x)+|\nabla w_{\sigma,\varepsilon}^A(x)|^q
                    &\le \gamma_{\sigma,\varepsilon}A |x|^{-\sigma-2s}+A^q\sigma^q|x|^{-(\sigma+1)q}\\
                    &\le (\gamma_{\sigma,\varepsilon}A+A^q\sigma^q)|x|^{-\sigma-2s}, \quad
                   |x|>1,
                \end{aligned}
            \]
            due to  
            $w_{\sigma.\varepsilon}^A(x)= Aw_{\sigma.\varepsilon}(x).$ Then the conclusion 
            follows taking $A$ small enough.                    
		\end{proof}
    
     \subsection{Comparison Principle and some pointwise inequalities} 
     	To conclude this section we show now two different versions of the Comparison Principle for our equations and some useful inequalities that will be needed later. Before that we want to mention that some other kind of comparison principles for different types of solutions in the nonlocal framework can be founded in, for instance, \cite[Lemma 2.9]{MR3341459}, \cite[Lemma 3.1]{BQ}, \cite[Theorem 3.2]{MR2868849}, \cite[Proposition 2.5]{MR3631323}, \cite[Lemma 6]{MR3713547} and \cite [Proposition 2.17]{MR2270163}. We begin with the next

     	\begin{theorem}\label{theorem:cp}
            Let $q>0,$ $u$ and $v$ be classical sup and subsolution of 
            \[
                (-\Delta)^s w
                + |\nabla w|^q=0 \quad\text{in } 
                \mathbb{R}^N\setminus B_{R_0},
            \]
            respectively. 
            If $u$ is positive in 
            $\mathbb{R}^N\setminus B_{R_0},$
            \[
                v(x)\to 0\quad\text{as }|x|\to\infty,
            \]
            and $u(x)\ge v(x)$ in $B_{R_0}$ 
            then $u(x)\ge v(x)$ in $\mathbb{R}^N.$
        \end{theorem}
		\begin{proof} 
       		Let $w(x)\coloneqq v(x)-u(x)\leq 0$ in $B_{R_0}$ and 
       		let us suppose that there exists $x_1\in \mathbb{R}^N\setminus 
       		B_{R_0}$ such that $w(x_1)> 0$. Since $u$ is positive in 
       		$\mathbb{R}^N\setminus B_{R_0}$, $v(x)\to 0$ as $|x|\to\infty,$ 
       		and $u(x)\ge v(x)$ in $B_{R_0}$  
       		there is $x_0\in \mathbb{R}^N\setminus B_{R_0}$ such that 
            \[
            	w(x_0)\ge w(x),\quad x\in\mathbb{R}^N,
            \]
            Thus
            \[
                0=\nabla w(x_0)=\nabla v(x_0)-\nabla u(x_0),
            \]
            and
            \[
                0\le(-\Delta)^s w(x_0)=(-\Delta)^s v(x_0)-(-\Delta)^s u(x_0)
                \le 0.
            \]
            Therefore $w(x)=w(x_0)$ for every $x\in\mathbb{R}^N$ 
            which is a contradiction, because
            $w(x_0)>0$ and $w(x)\le0$ in  $B_{R_0}.$ 
            Thus $w(x)\leq 0$ in $\mathbb{R}^N\setminus B_{R_0}$ as wanted.
        \end{proof}    
     	\begin{theorem}\label{theorem:cpp}
            Let $q>0,$ $f\colon\mathbb{R}\to\mathbb{R}$ be a continuous function and 
            $u$ and $v$ be such that
            \[
                (-\Delta)^s u
                + |\nabla u|^q\ge f(x)\ge (-\Delta)^s v
                + |\nabla v|^q \quad\text{in } 
                R_0<|x|<R,
            \]
        	in classical sense.
            If $u\ge v$ in 
            $B_{R_0}\cup \left(\mathbb{R}^N\setminus B_{R}\right)$,
            then $u(x)\ge v(x)$ in $\mathbb{R}^N.$
        \end{theorem}

	\medskip
    
    We finish this section with the next two lemmas that are a simply generalization of \cite[Lemmas 3 and 4]{MR3039209}. Since their proofs are analogous we omit them.
    
    \begin{lemma}\label{lemma:AGMQ3}
        Let $q>1,$ $p>0,$ $s>1/2$ and $h(R)$ be a positive decreasing function
        defined for $R>R_0$ verifying 
        \[
            h(2R)^p\le C\left(\dfrac{h(R)}{R^{2s}}+\dfrac{h(R)^q}{R^{q}}\right),
        \]
        for $R>R_0$ and some positive constant $C>0.$ Then.
        
        \begin{enumerate}[label={\alph*)}]
	        \item If  $0<p\le 1$ then for every 
	        $\theta<0$ there exists a positive constant $C$
	        such that
	        \[
	            h(R)\le CR^{\theta},
	        \] 
	        for every $R>R_0.$
	        
	        \item If $1<q<2s$ and $1<p<\frac{2s-1}{2s-q}q$ or $q\ge 2s$  and $p>1$, then for every
	        $\theta\in\left(-\frac{2s}{p-1},0\right)$ there exists a positive constant $C$
	        such that
	        \[
	            h(R)\le CR^{\theta},
	        \] 
	        for every $R>R_0.$

        \end{enumerate}
    \end{lemma}

    \begin{lemma}\label{lemma:AGMQ4}
        Let $q>1,$ $p>0,$ $s>1/2$ and $h(R)$ be a positive decreasing function
        defined for $R>R_0$ satisfying 
        \[
            h(R)^p\le C\left(\dfrac{h(R)}{R^{2s}}+\dfrac{h(R)^q}{R^{q}}\right),
        \]
        for $R>R_0$ and some positive constant $C>0.$ 
        
        \begin{enumerate}[label={\alph*)}]
	        
	        \item If $1<q<2s$ and $p=\frac{2s-1}{2s-q}q$ then  
	        there exits a positive constant $C$
	        such that
	        \[
	            h(R)\le CR^{-\displaystyle\frac{2s}{p-1}},
	        \] 
	        for every $R>R_0.$
	        
	        \item If $q> \frac{N}{N+1-2s}$ and $p= \dfrac{N}{N-2s}$ then 
	         there exits a positive constant $C$ such that
	        \[
	            h(R)\le CR^{-N{+}2s},
	        \] 
	        for every $R>R_0.$
	        
        \end{enumerate}
    \end{lemma}

\section{Nonexistence results}\label{nonexist}
	
	Throughout all this section, we will assume that $q>1$,
	$f\colon(0,\infty)\to\mathbb{R}$ is a continuous function, $\lambda>0$ and $u$ 
	be a  positive supersolution of \eqref{main} which does not blow up at infinity.
	
	Given $0<R_1<R_2,$ we define
	\begin{equation}\label{m}
		m(R_1,R_2)\coloneqq\min\left\{u(x)\colon x\in A(R_1,R_2)\right\},
	\end{equation}
	where
	\[
		A(R_1,R_2)\coloneqq\left\{x\colon R_1\le |x|\le R_2\right\}.
	\]

	Observe that $m(R_1,\cdot)$ is a {nonincreasing} 
	function in $(R_1,\infty)$ and 
	$m(\cdot,R_2)$ is a {nondecreasing} 
	function in $[0,R_2).$ 
	{As we mentioned in the introduction, unlike in the local case, it cannot be proved that 
	$m(R_1,R_2)=\min\{m(R_1), m(R_2)\}$ 
	where $m(R)=min\{u(x)\colon |x|=R\}$ is, in the local framework, 
	a monotone function for $R>R_1$.}
		
	\subsection{Preliminary results: some bounds for the $m$ function.} 
		In order to prove the nonexistence of positive supersolutions 
		 which do not blow up at infinity,
		we will get some previous auxiliar lemmas 	
		regarding with the function $m(R_1,R_2)$, $0<R_1<R_2$, defined in \eqref{m}. 
		More precisely we obtain 	
		upper and lower bounds for $m(0,R)$ that will be the key steps to obtain 
		the nonexistence results (see Lemmas \ref{lemma:aux2}, \ref{claims}
		and \ref{lemma:aux4}).

		\begin{lemma}\label{lemma:aux1}
			Let $\lambda>0,$ $f(t)>0$ in $(0,\infty)$ be a continuous function and $u$ be a positive 
			{classical} {supersolution} 
				of \eqref{main} {which does not blow up at infinity}. Then 
				\[
				    \lim_{R\to \infty}m(R,4R)=0,
				\]
				with $m$ given in \eqref{m}.	\end{lemma}
		\begin{proof}
			Let $R>R_0$ fixed but arbitrary and $\eta\in C^\infty(\mathbb{R})$ be such that
			
			\begin{minipage}{0.2\textwidth}
				    \[
					     \eta(t)\coloneqq
					            \begin{cases}
					                1 & \text{if } t\in[2,3],\\
					                0 & \text{if } t<1 \text{ or } t>4.
					            \end{cases}
					\]
			\end{minipage}
			\begin{minipage}{0.9\textwidth}
				\begin{center}
					    \begin{tikzpicture}
				            \draw[->] (0,0) -- (5,0) node[right] {$t$};
				            \draw[->] (0,0) -- (0,1.5) node[left] {$y$};
						    \draw[thick, blue,rounded corners=1mm] 
						    (0,0)--(1,0)--(2,1)--(3,1)--(4,0)--(5,0);
					        \draw (4,1.3) node[left] {$\eta(t)$};
					        \draw[densely dotted] (0,1) node[left] {$1$} -- (2,1);
					        \draw[densely dotted] (2,0) node[below] {$2$} -- (2,1);
					        \draw[densely dotted] (3,0) node[below] {$3$} -- (3,1);
				        \end{tikzpicture}
			  	\end{center}
			\end{minipage}

			We define 
			   \[
				    v(x)\coloneqq u(x)-m(2R,3R)\eta\left(\dfrac{|x|}{R}\right),\, R>0.
				\]
				Observe that,  there exists $x_R\in A(R,4R)$ such that 
				\[
				    v(x_R)\le v(x), \quad x\in\mathbb{R}^N.
				\] 
				Then
				\begin{align*}
					 0&\ge (-\Delta)^s v(x_R)=(-\Delta)^s u(x_R)-\dfrac{m(2R,{3}R)}{R^{2s}}
				        (-\Delta)^s\eta\left(\dfrac{|x_R|}{R}\right),\\
				     0&=\nabla v(x_R)=\nabla u(x_R)- \dfrac{m(2R,3R)}{R}\nabla
				     \eta\left(\dfrac{|x|}{R}\right)\dfrac{x_R}{|x_R|}.
				\end{align*}
				Thus, since $u$ is a positive supersolution of \eqref{main}, 
				we have that
				\[
				   \lambda f(u(x_R))\le (-\Delta)^s u(x_R)+|\nabla u(x_R)|^q\le 
				   C(\eta)\left(\dfrac{m(2R,3R)}{R^{2s}}+\dfrac{m(2R,3R)^q}{R^q}\right). 
				\]
				Therefore, since $m(2R,3R)$ is bounded we get
				\[
				   f(u(x_R))\to 0\quad\text{as }R\to\infty,
				\]
			   so, using the fact that $u(x_R)\leq C$, it follows that
				\[
				   u(x_R)\to 0\quad\text{as }R\to\infty.
				\]
				Thus, since $x_R\in A(R,4R)$ the desired conclusion follows. 
		\end{proof}
		
		\begin{corollary}\label{corrollary:aux1}
				Under the same hypothesis as in Lemma \ref{lemma:aux1}, if $m$ is given in \eqref{m}, then		for all $R_1\ge0$,
				\[
					\lim_{R\to \infty} m(R_1,R)=0.
				\]
				Moreover
				\[
					m(R_1,R)=m(0,R),
				\]
				for all  large enough $R.$ 
		\end{corollary}
		\begin{proof}
				Let $R_1\ge0.$ For any $R>4R_1,$ we have that
				\[
					m(R_1,R)\le m\left(\dfrac{R}{4},R\right).
				\]
				Then, by Lemma \ref{lemma:aux1}, we get
				\[
				    \lim_{R\to \infty}m\left(\displaystyle\frac{R}{4},R\right)=0,
				\]
				so, therefore,
				\begin{equation}\label{eq:limaux1}
					\lim_{R\to \infty} m(R_1,R)=0.
				\end{equation}
Finally, by the fact that
				\[
					m(0,R)=\min\{m(0,R_1),m(R_1,R)\}, \quad R>R_1,
				\]
				by \eqref{eq:limaux1}, taking $R$ large enough, we conclude 
				 \[
					m(0,R)=\min\{m(0,R_1),m(R_1,R)\}=m(R_1,R).
				\]    
		\end{proof}
		
		\begin{lemma}\label{lemma:aux2}
			Let $\lambda>0,$ $f(t)>0$ in $(0,\infty)$ be a continuous function that satisfies 
			\eqref{eq:f1}, and $u$ be a positive 
			classical supersolution of \eqref{main} which does not blow up at infinity.
			Then there exists a positive constant $C$ such that
			\[
				m(0,2R)^p\le C\left(\dfrac{m(0,R)}{R^{2s}}+\dfrac{m(0,R)^q}{R^q}\right),
			\]
			for $R$ large enough. That is, $m(0,R)$ satisfies the hypothesis of Lemma \ref{lemma:AGMQ3} 
			where $m$ was defined in \eqref{m}.
		\end{lemma}
		\begin{proof}	  
			Let $R_1>R_0$ and $R$ be large enough so that 
				\begin{equation}\label{extra}
				  m(0,R)=m(R_1,R),
				  \end{equation}
				  and
				$$
				   m(0,R_1)>m(R_1,R).
				$$
Given $\xi\in C^\infty(\mathbb{R})$  such that
				  
			\begin{minipage}{0.25\textwidth}
			 	\[
					 \xi(t)\coloneqq
					            \begin{cases}
					                1 & \text{if } |t|<1,\\
					                0 & \text{if } |t|>2,
					            \end{cases}
				\]
			\end{minipage}
			\begin{minipage}{0.85\textwidth}
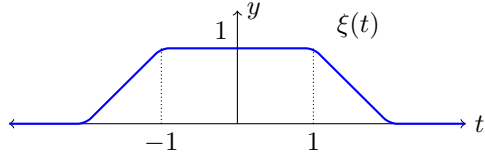
				    \begin{center}
					    \begin{tikzpicture}
				            \draw[<->] (-3,0) -- (3,0) node[right] {$t$};
				            \draw[->] (0,0) -- (0,1.5) node[right] {$y$};
						    \draw[thick, blue,rounded corners=1mm] 
						    (-3,0)--(-2,0)--(-1,1)--(1,1)--(2,0)--(3,0);
					        \draw (2,1.3) node[left] {$\xi(t)$};
					        \draw (0,1) node[above left] {$1$};
					        \draw[densely dotted] (-1,0) node[below] {$-1$} -- (-1,1);
					        \draw[densely dotted] (1,0) node[below] {$1$} -- (1,1);
				        \end{tikzpicture}
				    \end{center}
			 \end{minipage}
			 we define
			 \[
				  w(x)\coloneqq u(x)-m(R_1,R)\xi\left(\dfrac{|x|}{R}\right).
			 \]
			 Note that $w(x)>0$ in $B_{R_1}\cup(\mathbb{R}^N\setminus B_{2R}).$ 
			 Then there exits $x_R\in A(R_1,2R)$ such that 
			 \[
				  w(x_R)\le w(x), \quad x\in\mathbb{R}^N.
			 \] 
			From here, the argument proceeds as in the proof of 
			Lemma \ref{lemma:aux1} by using  \eqref{eq:f1}.
				
		\end{proof} 
	
		We will give now a lower bound for the function $m$ for all ranges of $q$, that is, we have the next.
	  	
	  	\begin{lemma}\label{claims} Under the same hypothesis as in Lemma \ref{lemma:aux2}, if 		
		    \begin{enumerate}
	 			\item[(i)] $1<q<\displaystyle\frac{N+2s}{N+1}$ 
		    			and $\sigma\ge \displaystyle\frac{N+2s}q-1$
		    			or $q=\displaystyle\displaystyle\frac{N+2s}{N+1}$ and $\sigma>N$,
	 			\item[(ii)] $\displaystyle\displaystyle\frac{N+2s}{N+1}<q< \displaystyle\frac{N}{N+1-2s},$ 
	 				and $N>\sigma>\displaystyle\frac{2s-q}{q-1},$
	 			\item [(iii)] $q\ge\displaystyle\frac{N}{N+1-2s}$ and $N>\sigma>N-2s,$
	 		\end{enumerate}
		    then there exists a positive constant $A$ such that
		    \begin{equation}\label{eq:b1}
		        m(0,R)\geq AR^{-\sigma},
		    \end{equation}
		    for $R$ large enough and $m$ given in \eqref{m}.
		\end{lemma}     

		\begin{proof}  
			{\it (i)}  Let us take  a positive real function 
			$\phi\in \mathcal{C}^2(\mathbb{R}^N)$ such that
			$\phi$ is radially symmetric and 
			decreasing in $|x|>1.$ 
			Assume also that there exists $\sigma>0$ such that 
			$\phi(x)<|x|^{-\sigma}$, $|\nabla\phi(x)|<c_0|x|^{-\sigma-1},$ and
			$|D^2\phi(x)|<c_0|x|^{\sigma-2},$ 
		    for $|x|$ large enough. By Lemma \ref{lemma:subsolution1} we have that
		    if $1<q<\tfrac{N+2s}{N+1}$ and $\sigma\ge \displaystyle\tfrac{N+2s}q-1$
		    or $q=\tfrac{N+2s}{N+1}$ and $\sigma>N$ then
		    for any $A>0$ small enough, the function $\phi_A(x)=A\phi(x)$
		    satisfies
		    \[
		        (-\Delta)^s\phi_A(x)+|\nabla \phi_A(x)|^q
		            \le 0
		            \quad \text{in }|x|>R, 
		    \]
		    for $R$ large enough. Moreover, we can take $A$  small enough so that
		    \[
		        \phi_A(x)\le u(x), \quad |x|\le R.
		    \]
		    Then, by Theorem \ref{theorem:cp}, we have
		    \[
		        \phi_A(x)\le u(x),\quad x\in\mathbb{R}^N.
		    \]
		    Thus the conclusion follows.
		    
		    \medskip
		    
		    \noindent {\it (ii)} By Lemma \ref{lemma:subsolution2} we can take 
		    $\varepsilon$ and $A$ small enough and 
		    	$R_1\gg\max\{R_0,1\}$ where $R_0$ is given in 
		    	\eqref{main}, so that
		    	\[
		        	w_{\sigma,\varepsilon}^A(x)\le u(x),\quad |x|\le R_1, 
		    	\]
		    	and 
		    	\[
		        	(-\Delta)^s w_{\sigma.\varepsilon}^A(x)
		        	+|\nabla w_{\sigma,\varepsilon}^A(x)|^q
		            \le 0,
		            \quad |x|>R_1,
		    	\]
			where $\sigma\in(\tfrac{2s-q}{q-1},N)$. Then, doing comparison again, 
			by Theorem \ref{theorem:cp}, we have
		     	\[
		        	w_{\sigma,\varepsilon}^A(x)
		        	\le u(x),\quad x\in\mathbb{R}^N,
		    	\]
		    	so \eqref{eq:b1} follows.
			
			\medskip
		
			\noindent {\it (iii)} The proof is similar to the case $(ii)$ considering 
			$\sigma\in(N-2s,N)$. 
	\end{proof}    	
		
	
	To conclude this section regarding with the estimates of the function $m$, we introduce some 
	lower and upper bounds for $m(0,R)$ that will be necessary in the critical case 
	$p=\tfrac{N}{N-2s}$.
\begin{lemma}\label{lemma:aux4} Under the same hypothesis of Lemma \ref{lemma:aux1} if 
		   $q>\displaystyle\tfrac{N}{N+1-2s}$ then
		    \[
		    	m\left(0,\dfrac{R}2\right)\le C m(0,R),
		    \]
		    for $R$ large enough and where $m$ was given in \eqref{m}.
		\end{lemma}
		
		\begin{proof}
			For the proof, we borrow some ideas of \cite[Lemma 4.2]{MR2739791}.
			In fact, given $R>R_0,$ we take
			\[
				\overline{R}(\varepsilon)=\overline{R}:=R\left[\dfrac{\varepsilon}{1+\varepsilon2^{-N+2s}}\right]^{\frac1{N-2s}},
			\]
			where $\varepsilon$ will be selected later. We first take $\varepsilon>0$ small enough such that $\overline{R}<R/2.$ We define the following radial functions

		    \begin{minipage}{0.5\textwidth}
		       \[
				w(x)\coloneqq
					\begin{cases}
						\overline{R}^{-N+2s} &\text{ if } 0<|x|<\overline{R},\\
						|x|^{-N+2s}  &\text{ if } \overline{R}\le t,
					\end{cases}
			\]
		    \end{minipage}
		    \begin{minipage}{0.4\textwidth}
		        \begin{center}
			        \begin{tikzpicture}
		                \draw[->] (0,0) -- (4,0) node[right] {$r$};
		                \draw[->] (0,0) -- (0,2) node[left] {$y$};
				        
			            \draw (3,1) node[left] {$w(r)$};
			            
			            \draw[densely dotted] (0.9,1.44594314594) -- (0.9,0) node[below] {$\overline{R}$};
			            \draw[densely dotted] (1.4,0.32594314594) -- (1.4,0) node[below] {$\frac{R}{2}$};			            				\draw[densely dotted] (2,0.08838834764) -- (2,0) node[below] {$2R$};
			             \draw[blue,scale=1] 
		            	plot[domain=.9:4] ((\x,{\x^(-3.5)});

				        \draw[thick, blue,rounded corners=1mm] 
				        (0,1.44594314594)--(0.9,1.44594314594);
		            \end{tikzpicture}
		        \end{center}
		    \end{minipage}
		    \\
		    and\\
		    \begin{minipage}{0.5\textwidth}
		       \[
				w_R(x)\coloneqq
					\begin{cases}
						w(x)&\text{ if } |x|\le 2R,\\
						(2R)^{-N+2s}  &\text{ if } 2R\le |x|,\\
					\end{cases}
			\]
		    \end{minipage}
		    \begin{minipage}{0.4\textwidth}
		        \begin{center}
			        \begin{tikzpicture}
		                \draw[->] (0,0) -- (4,0) node[right] {$r$};
		                \draw[->] (0,0) -- (0,2) node[left] {$y$};
				        
			            \draw (3,1) node[left] {$w_R(r)$};
			            
			            \draw[densely dotted] (0.9,1.44594314594) -- (0.9,0) node[below] {$\overline{R}$};
			            \draw[densely dotted] (1.4,0.31594314594) -- (1.4,0) node[below] {$\frac{R}{2}$};				            \draw[densely dotted] (2,0.08838834764) -- (2,0) node[below] {$2R$};
			             \draw[blue,scale=1] 
		            	plot[domain=.9:2] ((\x,{\x^(-3.5)});
		            	\draw[thick, blue,rounded corners=1mm] 
				        (2,0.08838834764 )--(4,0.08838834764);
				        \draw[thick, blue,rounded corners=1mm] 
				        (0,1.44594314594)--(0.9,1.44594314594);
		            \end{tikzpicture}
		        \end{center}
		    \end{minipage}
		    \\
			We set now
			\[
				\phi(x)\coloneqq m\left(0,\frac{R}2\right)\frac{w_R(x)-w(2R)}{w(\overline{R})-w(2R)},
			\]
			and we claim that, for any $\displaystyle\frac{R}2<|x|<2R$,
			\begin{equation}\label{eq:claim1}
				\begin{aligned}
					(-\Delta)^s\phi(x)+|\nabla\phi(x)|^q=&
					\dfrac{m\left(0,\frac{R}2\right)}{w(\overline{R})-w(2R)}(-\Delta)^sw_R(x)\\
					&+\left(\dfrac{m\left(0,\frac{R}2\right)}{w(\overline{R})-w(2R)}\right)^q
					\frac1{|x|^{(N+1-2s)q}}\\
					&\le0.
				\end{aligned}
			\end{equation}
			To check the previous claim we start by observing that,
			since $q>\tfrac{N}{N+1-2s},$ if $R>2,$ then for any $\displaystyle\frac{R}2<|x|<2R$ it follows that
			\begin{align*}
					(-\Delta)^s\phi(x)+|\nabla\phi(x)|^q<&
				\dfrac{m\left(0,\frac{R}2\right)}{w(\overline{R})-w(2R)}(-\Delta)^sw_R(x)\\
				&+\left(\dfrac{m\left(0,\frac{R}2\right)}{w(\overline{R})-w(2R)}\right)^q
				\frac1{|x|^{N}}.
			\end{align*}
			Then, to prove \eqref{eq:claim1} 
			is enough to get that there exists a positive constant $C$ such that 
			\begin{equation}\label{rsme}
				(-\Delta)^sw_R(x)\le-\dfrac{C}{|x|^{N}},
			\end{equation}
			for any $\displaystyle\frac{R}2<|x|<2R.$ For that we notice that, since $\eta(x)\coloneqq|x|^{-N+2s}$ is the fundamental solution of $(-\Delta)^s$, for every $\displaystyle\frac{R}2<|x|<2R$, 
			\begin{equation}\label{eq:w1}
				0=(-\Delta)^s\eta(x)\geq(-\Delta)^sw(x)+
				I_1(\varepsilon,x)+I_2(\varepsilon,x),
			\end{equation}
			where
			\begin{align*}
				I_1(\varepsilon,x) &\coloneqq\int_{B_{\overline{R}}(x)}
				\dfrac{\overline{R}^{-N+2s}-|x-y|^{-N+2s}}{|y|^{N+2s}}dy<0,\\
				I_2(\varepsilon,x) &\coloneqq\int_{B_{\overline{R}}(-x)}
				\dfrac{\overline{R}^{-N+2s}-|x+y|^{-N+2s}}{|y|^{N+2s}}dy<0.
			\end{align*}
			
			We choose now $\varepsilon$ small enough such that for any $|x|>\displaystyle\frac{R}2,$ and
			$y\in B_{\overline{R}}(-x)\cup B_{\overline{R}}(x)$ we have that $|y|\ge\displaystyle\frac{R}3.$
			Then, for any $R>|x|>\displaystyle\frac{R}2,$
			\begin{align*}
				I_1(\varepsilon,x)&\ge-\left(\dfrac{3}{R}\right)^{N+2s} C(N,s)\overline{R}^{2s}=
				-\dfrac{C(N,2s,\varepsilon)}{R^N}\ge-\dfrac{\widetilde{C}(N,2s,\varepsilon)}{|x|^N},
			\end{align*}
			where $\widetilde{C}(N,2s,\varepsilon)$ is a positive constant that goes to $0$ as $\varepsilon\to0.$
			Doing a similar computation for $I_2(\varepsilon,x)$, by \eqref{eq:w1} it follows that		
			\begin{equation}\label{eq:w2}
				(-\Delta)^sw(x)\le \dfrac{C(N,2s,\varepsilon)}{|x|^N},
			\end{equation}
			for every $\displaystyle\frac{R}2<|x|<2R$, where $C(N,2s,\varepsilon)$ is also a positive constant that goes to $0$ as $\varepsilon\to0.$
			
	On the other hand, for any $\displaystyle\frac{R}2<|x|<2R$ we get		
	\begin{equation}\label{eq:w3}
				(-\Delta)^sw_R(x)=(-\Delta)^sw(x)+E(\varepsilon,x),
			\end{equation}
			where
			\begin{align*}
				E(\varepsilon,x)&\coloneqq\int_{B^{c}_{2R}(x)\cup B^{c}_{2R}(-x)}
				\!\!\!\!\!\!\!\!\!\!\!\!!\!\!\!\!\!\!\!\!\!\!\!
				\dfrac{-w_R(x+y)-w_R(x-y)+|x+y|^{-N+2s}+|x-y|^{-N+2s}}{|y|^{n+2s}}dy\\
				&\le\int_{B^{c}_{5R}(0)}
				\dfrac{-2(2R)^{-N+2s}+|x+y|^{-N+2s}+|x-y|^{-N+2s}}{|y|^{n+2s}}dy.
			\end{align*}
			
			We observe that if $\displaystyle\frac{R}2<|x|<2R$ and $y\in\mathbb{R}^N\setminus B_{5R}(0)$ then 
			\[
				min\{|x+y|,|x-y|\}\ge\frac3{5}|y|.
			\]
			Thus, for any $x\in A(\displaystyle\frac{R}2,2R)$, we have
			\begin{equation}\label{eq:w4}
				E(\varepsilon,x)\le -\dfrac{C(N,s)}{R^{N}}\le-\dfrac{C(N,s)}{|x|^{N}}.
			\end{equation}
			Therefore by \eqref{eq:w2}-\eqref{eq:w4} we can select $\varepsilon$
			small enough such that \eqref{rsme} follows and, consequently, also the claim \eqref{eq:claim1}.
			
			Finally since $\phi(x)\leq u(x)$ if $x\in B_{R/2}\cup B_{2R}^{c}$, by Theorem \ref{theorem:cp}, we get that
			\[
				\phi(x)\le u(x),\quad x\in\mathbb{R}^N.
			\]
			Therefore, by taking the infimum in $0<|x|\leq R$, there exists a positive constant $C$ such that
			\[
		    	m\left(0,\dfrac{R}2\right)\le C m(0,R),
		    \]
		    for large enough $R.$ 
		\end{proof}
		
		To conclude this section we observe that by Lemmas \ref{lemma:aux2} and \ref{lemma:aux4}, we clearly deduce the following
		
		\begin{lemma}\label{lemma:aux5}
			 Under the same hypothesis of Lemma \ref{lemma:aux2}, if  
		   $q>\tfrac{N}{N+1-2s}$ then there exists 
			a positive constant $C$ such that
			\[
				m(0,R)^p\le C\left(\dfrac{m(0,R)}{R^{2s}}+\dfrac{m(0,R)^q}{R^q}\right).
			\]
			for $R$ large enough. That is, $m(0,R)$ satisfies the hypothesis of Lemma \ref{lemma:AGMQ4}.
		\end{lemma}

	\subsection{Nonexistence result in the subcritical case $(p<\tfrac{N}{N-2s})$}
		Using the technical lemmas showed in the previous section we can prove now the main 
		result of the work in the subcritical case, that is, Theorem \ref{theorm:subcritical1}.

		\begin{proofsubcritico} 
		    We suppose the contrary, that is, we assume that 
		    there exists a positive supersolution $u$ 
		    of \eqref{main} which does not blow up at infinity.

		    \bigskip

		    \noindent{\it{ Case 1}: $1<2s<N,$ $1<q\le \tfrac{N+2s}{N+1}$ 
		    and $0<p<\displaystyle\frac{(2s-1)q}{2s-q}.$}
		    
		    \medskip
		    
		    By \eqref{eq:b1}, Lemmas \ref{lemma:AGMQ3}, \ref{lemma:aux2} 
		    and \ref{claims}, 
		    we get that
		    \begin{itemize}
		        \item If $0<p\le1,$  either  $1<q<\tfrac{N+2s}{N+1}$ and 
		        $\sigma\ge \tfrac{(N+2s)}q-1$ or $q=\tfrac{N+2s}{N+1}$ and $\sigma>N$ 
		        then there is a positive constant $C$ such that
		        \[
		            AR^{-\sigma}\le m(0,R)\le C R^{\theta},\quad \theta<0,
		        \]
		        for $R$ large enough, which is a contradiction.
		        
		        \item If $q=\tfrac{N+2s}{N+1},$ 
		        $1<p<\tfrac{(2s-1)q}{2s-q}=\tfrac{N+2s}{N}$ 
		        and $\sigma>N$ 
		        then there is a positive constant $C$ such that
		            \[
		                AR^{-\sigma}\le m(0,R)\le C R^{\theta},\quad \theta\in 
		                \left(-\dfrac{2s}{p-1},0\right),
		            \]
		       if $R$ is large enough, which implies a contradiction as long as we can choose $\sigma<-\theta$. This is certainly possible due to the fact that
		        \[
		            -\dfrac{2s}{p-1}<-N\Longleftrightarrow
		            p<\dfrac{N+2s}{N}.
		        \]
		        
		        \item If $1<q<\tfrac{N+2s}{N+1},$ $1<p<\tfrac{(2s-1)q}{2s-q}$ and 
		        $\sigma\ge \displaystyle\frac{N+2s}q-1,$
		        then there is a positive constant $C$ such that
		            \begin{equation}\label{extra.1}
		                AR^{-\sigma}\le m(0,R)\le C R^{\theta},\quad \theta\in 
		                \left(-\dfrac{2s}{p-1},0\right),
		            \end{equation}
		        chossing $R$ large enough. Since $q<\dfrac{N+2s}{N+1}$ then 
		        \[
		        	\dfrac{2s-1}{2s-q}q<1+\dfrac{2s}{N+2s-q}q,
		        \] 
		        with is equivalent to 
		        \[
		        	 p<1+\dfrac{2s}{N+2s-q}q,
		       	\] 
		       	that implies 
		       	\[
		       		-\dfrac{2s}{p-1}<1-\dfrac{N+2s}q.
		       	\] 
		        Then we can take $\sigma<-\theta$ so that we get a contradiction with \eqref{extra.1}.
		    \end{itemize}
		
		    \bigskip
		    
		    \noindent{\it{Case 2}: $1<2s<N,$ 
		    $\tfrac{N+2s}{N+1}<q< \tfrac{N}{N+1-2s}$ and 
		    $0<p<\tfrac{(2s-1)q}{2s-q}.$}

		    \medskip

		    By \eqref{eq:b1} and Lemmas \ref{lemma:AGMQ3}, \ref{lemma:aux2} and 
		    \ref{claims} we obtain:
		    \begin{itemize}
		        \item If  $0<p\le1,$ $\tfrac{N+2s}{N+1}<q< \tfrac{N}{N+1-2s},$ 
		        and $N>\sigma>\tfrac{2s-q}{q-1},$
		        then, taking $R$ big enough, there exists a positive constant $C$ such that
		        \[
		            AR^{-\sigma}\le m(0,R)\le C R^{\theta},\quad \theta<0,
		        \]
		        which implies is a contradiction.
		        \item If 
		        $\tfrac{N+2s}{N+1}<q< \tfrac{N}{N+1-2s},$
		        $1<p<\tfrac{(2s-1)q}{2s-q},$ 
		        and $N>\sigma>\tfrac{2s-q}{q-1},$
		        then there exists a positive constant $C$ such that
		            \[
		                AR^{-\sigma}\le m(0,R)\le C R^{\theta},\quad
		               \theta\in \left(-\dfrac{2s}{p-1},0\right),
		            \]
		        for large enough $R.$ Since            \[
		            -\dfrac{2s}{p-1}<\dfrac{2s-q}{1-q}\Longleftrightarrow
		            p<\dfrac{2s-1}{2s-q}q,
		        \]
		        we deduce that we can choose $\sigma<-\theta$ so that the contradiction follows.
		    \end{itemize}

		    \bigskip
		    
		    \noindent{\it Case 3:  $1<2s<N,$ $q\ge\tfrac{N}{N+1-2s}$ and $0<p<\tfrac{N}{N-2s}.$}
		   Using  \eqref{eq:b1} and Lemmas \ref{lemma:AGMQ3}, \ref{lemma:aux2} and \ref{claims} it follows that
		   
		    \begin{itemize}
		        \item If $0<p\le1$ and {$\tfrac{N}{N+1-2s}\le q$} 
		      , if $R$ is large enough, there exists a positive constant $C$ such that
		        \[
		            AR^{-\sigma}\le m(0,R)\le C R^{\theta},\quad\theta<0,
		        \]
		        which implies a contradiction.
		        \item If $1<p<\tfrac{N}{N-2s}$ 
		           and $q\ge 2s$  then there is a positive constant $C$ such that
		            \[
		                AR^{-\sigma}\le m(0,R)\le C R^{\theta},\quad 
		                \theta\in \left(-\dfrac{2s}{p-1},0\right),
		            \]
		        choosing $R$ big enough which implies a contradiction if we choose $\sigma$ close enough to $N-2s.$
		        \item Finally we show up that if 
		        $1<p<\tfrac{N}{N-2s}$ and $\tfrac{N}{N+1-2s}\le q< 2s$ then
		            $1<p<\frac{2s-1}{2s-q}q.$ Therefore once again
		            there exists a positive constant $C$ such that
		            \[
		                AR^{-\sigma}\le m(0,R)\le C R^{\theta},\quad 
		                \theta\in \left(-\dfrac{2s}{p-1},0\right),
		            \]
		        as long as $R$ is large enough. As before this implies a contradiction 
		        taking $\sigma$ close enough to $N-2s.$
		    \end{itemize}
		    \hfill$\square$
		\end{proofsubcritico}

	\subsection{Nonexistence results in the critical case ($p=\displaystyle\frac{N}{N-2s}$)}} 
	Before proving the nonexistence result regarding with the critical case (see Theorem \ref{theorem:cc}),
	we need the following auxiliar 
		\begin{lemma}\label{lemma:aux6}
			 Under the same hypothesis of Lemma \ref{lemma:aux2}, if $p=\tfrac{N}{N-2s}$, 
		$q>\tfrac{N}{N+1-2s}$ and $f$ is a nondecreasing function, there exists $R_1>R_0$ 
			and a positive constant $C$ such that
			\[
				u(x)>\dfrac{C}{|x|^{N-2s}},
			\]
			for any $|x|>R_1.$
		\end{lemma}
		\begin{proof}
			Set
			\[
				w(x)\coloneqq\begin{cases}
					1 &\text{if } |x|\le 1,\\
					|x|^{-N+2s}&\text{if } |x|\ge 1,
				\end{cases}
			\]  
			By \eqref{eq:w2}, replacing $\overline{R}$ by $1$, we know that 
			\begin{equation}\label{eq:auxw1}
				(-\Delta)^s w(x)\le \dfrac{C}{|x|^{N+2s}},\quad C>0,
			\end{equation}
			for every $|x|>1$. For any $R_2>R_1\gg 1,$ we define
			$$
				\phi(x):=
				\begin{cases}
					0 &\text{if } |x|<1 \text{ or } |x|>2R_2,\\
					m(0,R_1)\frac{w(x)-w(R_2)}{w(1)-w(R_2)}&\text{if } 1< |x|\le 2R_2,
				\end{cases}
			$$
			that satisfies
			$$(-\Delta)^s \phi(x)+|\nabla \phi(x)|^q< C\left(\dfrac{1}{|x|^{N+2s}}
				+\dfrac{1}{|x|^{(N+1-2s)q}}\right)\leq C\frac{1}{|x|^{\sigma p}},$$			
			for some $C=C(N,s,R_1)$, $\sigma(q,N)\in(N-2s,N)$ and every $x\in A(R_1,2R_2).$ 
			Observe that the existence of $\sigma$ in this precise range comes from the hypothesis 
			$q>\tfrac{N}{N-2s+1}$. 
			
			On the other hand, by Lemma \ref{claims} {\it iii)}, 
			there exist $R>R_1$ and a positive constant $A$ such that
			\[
				u(x)\ge\dfrac{A}{|x|^{\sigma}},\quad x\in B_{R}.
			\]
			Thus, since $f$ is nondecreasing and verifies \eqref{eq:f1}, there exist $R>R_0$
			large enough and a positive constant $\tilde{A}$ such that  
			\[
				(-\Delta)^s u(x)+|\nabla u(x)|^q\ge\dfrac{\tilde{A}}{|x|^{\sigma p}},\quad x\in B_{R}\setminus B_{R_0}.
			\]
Then, since $\phi(x)\le u(x)$ for any $x\in\mathbb{R}^N\setminus A(R_1,R_2)$, by Theorem \ref{theorem:cpp} there exists $K>0$ such that 
			\[
				K\phi(x)\le u(x),\quad x\in\mathbb{R}^N.
			\]
			Passing to the limit as $R_2\to\infty,$ we get
			\[
				\frac{C}{|x|^{N-2s}}\le u(x),\quad  |x|>R_1.
			\]
			as wanted
		\end{proof}
		Now we can prove the main result of this subsection, that is, we can give the
		
		\begin{proofcritical}
			Let us define the radial function
			\[
				\Gamma(x)\coloneqq\dfrac{\log(1+|x|)}{|x|^{N-2s}},
			\]
			that is decreasing in $(r_0,\infty)$ for some $r_0>0$. By \cite[Lemma 6.1]{MR2739791}, there exists a positive constant $C$ such that
			\[
				(-\Delta)^s {\Gamma}(x)\le\frac{C}{|x|^{N}},\quad x\in\mathbb{R}^N\setminus\{0\}.
			\]
Given $\min\{r_0,R_0\}<\varepsilon< R_1<R_2,$ we define now
			\[
				\phi(x)\coloneqq  
					\begin{cases}
						m(0,R_1)\frac{w(x)-w(R_2)}{w(\varepsilon)-w(R_2)} &\text{if } |x|\le R_2,\\
						0&\text{if } |x|>R_2,
					\end{cases}
			\]
			where here
			\[
				w(x)\coloneqq
				\begin{cases}
	 				\Gamma(\varepsilon)&\text{if } |x|\le\varepsilon,\\
	 				\Gamma(x)&\text{if } |x|> \varepsilon.
				\end{cases}
			\]
			Since it can be proved (see \cite[page 2734]{MR2739791}) that $(-\Delta)^sw$ also satisfies the same kind of upper bound than $(-\Delta)^{s}\Gamma$ for $x\in A(R_1,R_2)$ then there exists $C>0$, a constant independent of $R_2$, such that,
			\[
				(-\Delta)^s\phi(x)\le\dfrac{m(0,R_1)}{w(\varepsilon)-w(R_2)}\dfrac{C}{|x|^N},\quad x\in A(R_1,R_2).			\]
			Thus if $x\in A(R_1,R_2),$ we get
			\begin{align*}
				(-\Delta)^s\phi(x)&+|\nabla \phi(x)|^q \le
				\dfrac{m(0,R_1)}{w(\varepsilon)-w(R_2)}\dfrac{C}{|x|^N}\\
				&+\left(\dfrac{m(0,R_1)}{w(\varepsilon)-w(R_2)}\right)^q
				\left[
					\dfrac{|x|^{-N+2s}}{1+|x|}+\dfrac{N-2s}{|x|^{N+1-2s}}\log(1+|x|)
				\right]^q\\
				&\le C\left[\dfrac{m(0,R_1)}{w(\varepsilon)-w(R_2)}\dfrac{1}{|x|^N}
				+\left(\dfrac{m(0,R_1)}{w(\varepsilon)-w(R_2)}\right)^q\dfrac{\log(1+|x|)^q}{|x|^{(N+1-2s)q}}
				\right],
			\end{align*}
			where $C$ is a positive constant independent of $R_2.$ Moreover, since  
			\[
				\dfrac{m(0,R_1)}{w(\varepsilon)-w(R_2)}\to\dfrac{m(0,R_1)}{w(\varepsilon)},\quad R_2\to\infty,
			\]
			taking $R_2$ big enough it follows that
			\begin{equation}\label{apple}
				(-\Delta)^s\phi(x)+|\nabla \phi(x)|^q \le
				 C\left[\dfrac{1}{|x|^N}
				+\dfrac{\log(1+|x|)^q}{|x|^{(N+1-2s)q}}
				\right],
			\end{equation}
			for every $x\in A(R_1,R_2)$. We show up that, since $q>\tfrac{N}{N+1-2s},$ it is clear that
			$\log(1+|x|)\leq|x|^{\beta}$ with
			$$\beta=\frac{q(N-2s+1)-N}{q}>0.$$			
			Using that fact, and the hypothesis on $f$, from \eqref{apple} we get that 
			\[
				(-\Delta)^s\phi(x)+|\nabla \phi(x)|^q 
				\le f\left(\dfrac{K}{|x|^N}\right), \quad x\in A(R_1,R_2),			\]
			with $K$ independent of $R_2$ as long as {$R_1$ is large enough.} Using now the Lemma \ref{lemma:aux6} by comparison we conclude the existence of $C>0$ such that
			\begin{equation}\label{eq:desaux1}
				C\phi(x)\le u(x),
			\end{equation}
			for every $x\in A(R_1,R_2).$
			
			On the other hand, by Lemmas \ref{lemma:aux5} and \ref{lemma:AGMQ4}{\it (b)}, 
			for $R$ large enough, there exists a positive
			constant $K$ such that
			\begin{equation}\label{eq:desaux2}
				u(x)\le \dfrac{K}{|x|^{N-2s}},\quad x\in B_{R}.
			\end{equation}
Finally, by \eqref{eq:desaux1} and \eqref{eq:desaux2}, we get  
			\[
				C\dfrac{\log(1+|x|)}{|x|^{N-2s}}\le u(x)\le \dfrac{K}{|x|^{N-2s}}, \quad x\in A(R_1,R_2),
			\]
			for  $R_1$ large enough, where $C$ and $K$ are positive constants independent of $R_2.$ Therefore 
			\[
				C\dfrac{\log(1+|x|)}{|x|^{N-2s}}
				\le u(x)\le \dfrac{K}{|x|^{N-2}}, \quad x\in \mathbb{R}^N\setminus B_{R_1},
			\]
			that clearly implies a contradiction.  \hfill$\square$
		\end{proofcritical}
\section{Existence of supersolutions}\label{exsol}
    Unlike to the previous section, positive supersolutions can be constructed when we consider $f(u)=\lambda u^p$ in \eqref{main} for some values of $\lambda$. In fact we give the

    \begin{proofexistencia}
    	Let $\varphi_{\sigma}=\varphi\in \mathcal{C}^2(\mathbb{R}^N)$ be a positive real function that is radially 
    	symmetric and decreasing for every $|x|>1,$ such that $\varphi(x)\le |x|^{-\sigma},$ 
    	$|D \varphi(x)|\ge c_0 |x|^{-\sigma-1}$ and 
    	$|D^2\varphi(x)|<\tilde{c_0}|x|^{-\sigma-2}$ 
    	for some $\sigma>0$ and for $|x|$ large enough (see Lemma \ref{lema.BV}).
    	We define
    	\[
    		\psi(x,R,\sigma):=\varphi\left(\dfrac{R}{R_0}x\right).
    	\]
    	Then, by Lemma \ref{lema.BV}, there exists $\tilde{R}>0$ such that   
    	\begin{equation}\label{eq.aux.oc}
    		\begin{aligned}
    			(-\Delta)^s\psi(x,R,\sigma)
    			&=\left(\dfrac{R}{R_0}\right)^{2s}(-\Delta)^s\varphi\left(\dfrac{R}{R_0}x\right)\\
    			&\ge
    			\begin{cases}
					-\left(\dfrac{R}{R_0}\right)^{-\sigma}\dfrac{c_1}{|x|^{\sigma+2s}} 
					&\text{if }\sigma<N,\\[8pt]
					-\left(\dfrac{R}{R_0}\right)^{-N}\dfrac{c_2\log(|x|)}{|x|^{N+2s}}&\text{if }\sigma=N,\\[8pt]
					-\left(\dfrac{R}{R_0}\right)^{-N}\dfrac{c_3}{|x|^{N+2s}}&\text{if }\sigma>N,
				\end{cases}
    		\end{aligned}
    	\end{equation} 
    	for all $R>\tilde{R}$ and $|x|>R_0.$ 
    	
    	We now split the rest of the proof in six different cases.
    	
    	\noindent{\it Case 1.} $p=\dfrac{N}{N-2s},$  $0<q<1.$ 
    		
    		Taking $\sigma=N-2s,$ by \eqref{eq.aux.oc}, for any $R>\tilde{R}$ and $|x|>R_0,$ we get
    			\begin{align}\label{alboran}
    				(-\Delta)^s&\psi(x,R,N-2s)+|\nabla \psi(x,R,N-2s)|^q\nonumber\\
    				&\ge\left(\dfrac{R}{R_0}\right)^{(2s-N)q}\dfrac{1}{|x|^{N}}
    				\left[
    					-\left(\dfrac{R}{R_0}\right)^{(2s-N)(1-q)}c_1+c_0^q
    					{|x|^{N-(N+1-2s)q}}
    				\right]\nonumber\\
    				&\ge\left(\dfrac{R}{R_0}\right)^{(2s-N)q}\dfrac{1}{|x|^{N}}
    				\left[
    					-\left(\dfrac{R}{R_0}\right)^{(2s-N)(1-q)}c_1+
    					c_0^q{R_0^{N-(N+1-2s)q}}
    				\right].
    			\end{align}
			Since $1-q>0$ it is possible to take $R>\widetilde{R}$ big enough such that
    			\begin{align*}
    				(-\Delta)^s\psi&(x,R,N-2s)+|\nabla \psi(x,R,N-2s)|^q\\
    				&\ge\left(\dfrac{R}{R_0}\right)^{2sq+N(1-q)}\psi^p(x,R,N-2s)
    				\dfrac{c_0^q{R_0^{N-(N+1-2s)q}}}2\\
    				&\ge\lambda\psi^p(x,R,N-2s),
    			\end{align*}
    	for every $\lambda>0,$ and $|x|>R_0$.

	    	\bigskip
    	
    	\noindent{\it Case 2.} $p=\dfrac{N}{N-2s},$  and $q=1.$ 
   			
			Doing the same as {\it Case 1} for $|x|>R_0$   we get		
			\begin{align*}
    				(-\Delta)^s&\psi(x,R,N-2s)+|\nabla \psi(x,R,N-2s)|\\
    				&\ge\left(\dfrac{R}{R_0}\right)^{2s-N}\dfrac{1}{|x|^{N}}
    				\left[
    					-c_1+
    					c_0^q{|x|^{2s-1}}
    				\right]\\
    				&\ge\left(\dfrac{R}{R_0}\right)^{2s}\psi^p(x,R,N-2s)
    				\left[
    					-c_1+
    					c_0{R_0^{2s-1}}
    				\right].
    			\end{align*}
    			As long as $R_0>\tilde{R}_0$ such that 
    			\[	
    				-c_1+c_0{R_0^{2s-1}}>1,
    			\]
    			then it is possible to consider $R>\bar{R}$ big enough in order to have that and 
  
    			\[
    				(-\Delta)^s\psi(x,R,N-2s)+|\nabla \psi(x,R,N-2s)|
    				\ge\lambda\psi^p(x,R,N-2s),
    			\]
    		for every $\lambda>0$ and $|x|>R_0$.
    	\bigskip
    	
    	\noindent{\it Cases 3.} $p=\dfrac{N}{N-2s},$ and $1<q\le\dfrac{N}{N+1-2s}.$ 
    	
    		We will consider now $u(x)=A\psi(x,R,N-2s)$ where $A$ is a positive constant that 
    		will be chosen in a suitable way later. Repeating the computations done in \eqref{alboran} we get that
    		\begin{align*}
    				&(-\Delta)^su(x)+|\nabla u(x)|^q\\
    				&\ge\left(\dfrac{R}{R_0}\right)^{2s-N}\dfrac{A}{|x|^{N}}
    				\left[
    					-c_1+A^{q-1}c_0^q \left(\frac{R}{R_0}\right)^{(2s-N)(q-1)}
    				\right]\\
    				&\ge\left(\dfrac{R}{R_0}\right)^{2s}A^{1-p}u^p(x)
    				\left[
    					-c_1+A^{q-1}c_0^q\left(\frac{R}{R_0}\right)^{(2s-N)(q-1)}
    				\right],
    			\end{align*}
			as long as  $R>\tilde{R}$ and $|x|>R_0$.  Now we choose the positive constant $A$ so that, for instance,
    			\[
    				-c_1+A^{q-1}c_0^q \left(\frac{R}{R_0}\right)^{(2s-N)(q-1)}=1.
    			\]
    			Therefore, taking
    			\[
    				\lambda_0=\left(\dfrac{R}{R_0}\right)^{2s}A^{1-p},
    			\] 
    			we have that for any $0<\lambda<\lambda_0,$ 
    			 and $|x|>R_0,$ $u(x)$ is the suitable supersolution we were looking for.

    	\bigskip
    	
    	\noindent{\it Case 4.} $\tfrac{2s-1}{2s-q}q\le p<\tfrac{N}{N-2s},$  and 
       		$1<\tfrac{N+2s}{N+1}<q<\tfrac{N}{N+1-2s}.$
       		
       		We take $0<\sigma=\tfrac{2s-q}{q-1}$ and  $u(x)=A\psi(x,R,\sigma)$ 
       		where $A$ is a positive constant that 
    		will be chosen later. We observe that since $\tfrac{N+2s}{N+1}<q,$ then $\sigma<N$ so that, since $\sigma+2s=q(\sigma+1)$ by \eqref{eq.aux.oc} we get       		
    			\begin{align*}
    				(-\Delta)^s&u(x)+|\nabla u(x)|^q\\
    				&\ge\left(\dfrac{R}{R_0}\right)^{-\sigma}\dfrac{1}{|x|^{(\sigma+1)q}}A
    				\left[
    					-c_1+
    					c_0^{q}A^{q-1}\left(\dfrac{R_0}{R}\right)^{2s-q}
    				\right]\\
    				&\ge A^{1-p}\left(\dfrac{R}{R_0}\right)^{2s}u^{\frac{\sigma+1}{\sigma}q}(x)
    				\left[
    					-c_1+
    					c_0^{q}A^{q-1}\left(\dfrac{R_0}{R}\right)^{2s-q}
    				\right].\\
    			\end{align*}
			for every $R>\tilde{R}$ and $|x|>R_0$. Choosing now $A$ so that
    			\[
    				-c_1+
    					c_0^{q}A^{q-1}\left(\dfrac{R_0}{R}\right)^{2s-q}=1,
    			\]
    			we obtain 
    			\[
    				\lambda_0=\left(\dfrac{R}{R_0}\right)^{2s}A^{1-p}\simeq\left(\frac{R}{R_0}\right)^{\frac{2s-q}{q-1}(1-p)+2s},
    			\] 
    			Therefore, for every $0<\lambda<\lambda_0$ and $|x|>R_0$, we get
    			\[
    				(-\Delta)^s u(x)+|\nabla u(x)|^q\ge \lambda 
    				u^p(x),
    			\]
			as wanted
    	\bigskip
    	
    	\noindent{\it Case 5.} $\tfrac{N+2s}{N+2s-q}q\le p<\tfrac{N}{N-2s},$ and
       		$0<q<\tfrac{N+2s}{N+1},$ $q\neq1.$ 
       		
       		We consider now $u(x)=A\psi(x,R,\sigma)$ with $\sigma=\tfrac{N+2s-q}{q}$
       		and $A>0$ to be chosen. Since $0<q<\tfrac{N+2s}{N+1}$  then $\sigma >N$ so that using again 	\eqref{eq.aux.oc} it follows that
    			\begin{align*}
    				(-\Delta)^s&u(x)+|\nabla u(x)|^q\\
    				&\ge\left(\dfrac{R}{R_0}\right)^{-N}\dfrac{A}{|x|^{(\sigma+1)q}}
    				\left[
    					-c_3+
    					c_0^{q}A^{q-1}\left(\dfrac{R_0}{R}\right)^{2s-q}
    				\right]\\
    				&\ge\left(\dfrac{R}{R_0}\right)^{2s}
    				A\psi^{\frac{\sigma+1}{\sigma}q}(x,R,\sigma)
    				\left[
    					-c_3+
    					c_0^{q}A^{q-1}\left(\dfrac{R_0}{R}\right)^{2s-q}
    				\right].\\
    			\end{align*}
			for $R>\tilde{R}$ and $|x|>R_0$   by using the fact that $(\sigma+1)q=N+2s$. Choosing the positive constant $A$ so that
    			\[
    				-c_3+
    					c_0^{q}A^{q-1}\left(\dfrac{R_0}{R}\right)^{2s-q}=1,
    			\]
    			and taking
    			\[
    				\lambda_0=\left(\dfrac{R}{R_0}\right)^{2s}A^{1-p},
    			\] 
    			we have 
    			\[
    				(-\Delta)^s u(x)+|\nabla u(x)|^q\ge \lambda 
    				u^p(x),
    			\]
    		for every $0<\lambda<\lambda_0,$ and $|x|>R_0.$       	
		\bigskip
       	
       	  In the case of $q=1$ we will take $\sigma={N+2s-1}$ obtaining
       	
    			\begin{align*}
    				(-\Delta)^s&\psi(x,R,\sigma)+|\nabla \psi(x,R,\sigma)|\\
    				&\ge\left(\dfrac{R}{R_0}\right)^{-N}\dfrac{1}{|x|^{N+2s}}
    				\left[
    					-c_3+
    					c_0\left(\dfrac{R_0}{R}\right)^{2s-1}
    				\right]\\
    				&\ge\left(\dfrac{R}{R_0}\right)^{2s}
    				\psi^{\frac{N+2s}{N+2s-1}}(x,R)
    				\left[
    					-c_3+
    					c_0\left(\dfrac{R_0}{R}\right)^{2s-1}
    				\right],\\
    			\end{align*}
			for $R>\tilde{R}$ and $|x|>R_0$.  Taking $R_0>\widetilde{R}_{0}$ big enough such that
    			\[
    				-c_3+
    					c_0\left(\dfrac{R_0}{R}\right)^{2s-1}\ge 1,
    			\]
    			and
    			\[
    				\lambda_0=\left(\dfrac{R}{R_0}\right)^{2s},
    			\] 
    			we conclude that $u(x)$ is a supersolution of \eqref{eq:cc1} for every $0<\lambda<\lambda_{0}.$

    	\noindent{\it Case 6.} $\tfrac{N+2s}{N}< p<\tfrac{N}{N-2s},$ 
       		$q=\tfrac{N+2s}{N+1}.$
       		
		In this last case we observe that
       		$$\mbox{$\sigma+2s>(\sigma+1)q$ if and only if $\sigma<N$},$$ 
		and 
       		$$\tfrac{\sigma+1}\sigma q\to \tfrac{N+2s}N\quad \text{as } \sigma\to N.$$ 
		Then given $p>\tfrac{N+2s}{N}$ there exists 
       		$\sigma_p\in(0,N)$ such that $\tfrac{\sigma_p+1}{\sigma_p}q<p.$ Considering $u(x)=A\psi(x,R,\sigma_p)$, with $A$ chosen later, for every $R>\tilde{R}$ and $|x|>R_0,$ we get
    			\begin{align*}
    				&(-\Delta)^su(x)+|\nabla u(x)|^q\\
    				&\ge\left(\dfrac{R}{R_0}\right)^{-\sigma_p}\dfrac{A}{|x|^{(\sigma_p+1)q}}
    				\left[
    					-\dfrac{c_1}{|x|^{\sigma_p+2s-(\sigma_p+1)q}}+
    					c_0^{q}\left(\dfrac{R_0}{R}\right)^{-\sigma_p (q-1)} A^{q-1}
    				\right]\\
    				&\ge\left(\dfrac{R}{R_0}\right)^{(\sigma_p+1)q-\sigma_p}\!\!
    				A\psi^{\frac{\sigma_p+1}{\sigma_p}q}(x,R,\sigma_p)
    				\left[
    					-\dfrac{c_1}{R_0^{\sigma_p+2s-(\sigma_p+1)q}}+
    					c_0^{q}\left(\dfrac{R_0}{R}\right)^{\sigma_p(q-1)}A^{q-1}
    				\right].
    			\end{align*}
      			Taking $A>0$ in order to guarantee that
    			\[
    				-\dfrac{c_1}{R_0^{\sigma_p+2s-(\sigma_p+1)q}}+
    					c_0^{q}\left(\dfrac{R_0}{R}\right)^{\sigma_p(q-1)}A^{q-1}=1,
    			\]
    			and
    			\[
    				\lambda_0=\left(\dfrac{R}{R_0}\right)^{(\sigma_p+1)q-\sigma_p}A^{1-p},
    			\] 
			the conclusion follows.
    		
    \end{proofexistencia}

\medskip
   
    {\bf Acknowledgements.}
B. B. was partially supported by AEI grant MTM2016-80474-P and Ram\'on y Cajal fellowship RYC2018-026098-I (Spain). L. DP. was partially supported by the European Union's Horizon 2020 research and innovation 
program under the Marie Sklodowska-Curie grant agreement No 777822.

\bibliographystyle{abbrv}
\bibliography{Bibliografia}

\end{document}